\documentclass[12pt]{scrartcl}
\usepackage[english]{babel}
\usepackage[utf8x]{inputenc}
\usepackage[T1]{fontenc}
\usepackage{lmodern}
\usepackage{amsfonts}
\usepackage{amsmath}
\usepackage{amssymb}
\usepackage{amsthm}
\usepackage{mathrsfs}
\usepackage[colorlinks=false,urlcolor=blue]{hyperref}
\usepackage{listings}
\usepackage{graphicx}
\usepackage[shortlabels]{enumitem}
\usepackage{mathtools}
\usepackage{multirow}
\usepackage{color}
\usepackage{comment}

\addtokomafont{section}{\rmfamily}
\addtokomafont{subsection}{\rmfamily}
\addtokomafont{subsubsection}{\rmfamily}
\addtokomafont{paragraph}{\rmfamily}
\usepackage{indentfirst}

\swapnumbers
\theoremstyle{plain}
\newtheorem{theo}[paragraph]{Theorem}
\newtheorem{lemm}[paragraph]{Lemma}

\theoremstyle{definition}
\newtheorem{defi}[paragraph]{Definition}
\newtheorem{rema}[paragraph]{Remark}

\counterwithin{paragraph}{section}
\counterwithout{paragraph}{subsection}
\counterwithout{paragraph}{subsubsection}
\counterwithin{equation}{section}

\newcommand{\R}{\mathbb{R}}%real numbers
\newcommand{\N}{\mathbb{N}}%natural numbers
\newcommand{\Z}{\mathbb{Z}}%integers
\makeatletter
\@ifundefined{sl}{%
}{
}
\makeatother
\newcommand{\rmA}{\mathrm{A}}
\newcommand{\rmB}{\mathrm{B}}
\newcommand{\rmC}{\mathrm{C}}
\newcommand{\rmD}{\mathrm{D}}
\newcommand{\rmE}{\mathrm{E}}
\newcommand{\rmF}{\mathrm{F}}
\newcommand{\rmG}{\mathrm{G}}
\title{\rmfamily Balanced subsets in root systems}
\author{Andrei Moroianu\footnote{Université Paris-Saclay, CNRS,  Laboratoire de mathématiques d'Orsay, 91405 Orsay, France.}\ \footnote{Institute of Mathematics ``Simion Stoilow'' of the Romanian Academy, 21 Calea Grivitei, 010702 Bucharest, Romania.}, Paul Schwahn\footnote{Universidade Estadual de Campinas, IMECC, Rua Sérgio Buarque de Holanda 651, 13083-859 Campinas-SP, Brazil.}}
\date{}

\begin{document}

\maketitle
%{\let\thefootnote\relax\footnotetext{$^\ast$...}}

\begin{abstract}
\noindent
Balanced and well-balanced subsets of the set of positive roots of compact Lie algebras arise naturally in problems related to Hermitian and spin geometry. In this paper we compute the maximal and minimal size of well-balanced subsets in all simple root systems.

\medskip


\medskip

\end{abstract}

\section{Introduction}
\label{sec:intro}

Throughout this article, let $R$ be a (reduced) root system and $R^+\subset R$ a choice of positive roots.

\begin{defi}
A subset $S\subset R^+$ is called:
\begin{enumerate}[(i)]
 \item \emph{balanced} if it has a vanishing signed combination, that is, there exist signs $s_\alpha=\pm1$ such that $\sum_{\alpha\in S}s_\alpha\alpha=0$, or equivalently if $S$ can be partitioned into two subsets of vectors having the same sum;
 \item \emph{strongly orthogonal} if any two elements $\alpha,\beta\in S$ are strongly orthogonal, that is, neither $\alpha+\beta$ nor $\alpha-\beta$ is a root;
 \item \emph{well-balanced} if $S$ is balanced and $R^+\setminus S$ is strongly orthogonal.
\end{enumerate}
\end{defi}

Specifying a balanced subset of $R^+$ is equivalent to picking a subset $S$ of the set of \emph{all} roots with the property that $\sum_{\alpha\in S}\alpha=0$ and $S\cap(-S)=\emptyset$.

Strongly orthogonal sets of roots show up in many representation-theoretic, combinatorial, or geometric problems, e.g.~\cite{Har,KW,Take,Kane,Joy,AK1,AK2,BP,BGC}. Thus they have been widely investigated, in particular by Agaoka--Kaneda \cite{AK} who classified the maximal such subsets in simple root systems. The terminology \emph{strongly orthogonal} is motivated by the fact that strongly orthogonal roots are always orthogonal \cite[\S VI.1.3]{Bour}.

Balanced and well-balanced subsets also turn out to be useful in differential geometry. Artacho--Semmelmann \cite{AS} classified the simple root systems with the property that the entire set of positive roots is balanced (they call it \emph{strongly linearly dependent}) and used this to determine which of the full flag manifolds carry invariant spinors.

In \cite{torsion} we use well-balanced sets of roots to construct left-invariant almost Hermitian structures of Gray-Hervella type $\mathcal{W}_1\oplus\mathcal{W}_3$ (semi-Kähler with skew torsion) on compact Lie groups. Using this construction we classify all complete almost Hermitian manifolds carrying a Hermitian connection with parallel, skew-symmetric and closed torsion. Of particular interest for our construction are those well-balanced subsets $S$ with cocardinality $|R^+\setminus S|\neq1$. This motivates us to study the possible cardinality of such sets.

In this article we determine the minimal and maximal possible cardinality of well-balanced subsets of positive root systems. In a reducible root system, well-balanced subsets are always the union of well-balanced subsets of the factors, so we restrict our attention to simple root systems.
\medskip

\paragraph*{Acknowledgments.} 
A.M.~was partially supported by the PNRR-III-C9-2023-I8 grant \textbf{CF 149/31.07.2023} (Conformal Aspects of Geometry and Dynamics).

P.S.~was supported by FAPESP project \textbf{2024/08127-4}, part of the BRIDGES collaboration, and received partial support by FAPESP grant \textbf{2024/19272-5}.

\section{Preliminaries}

For any $l\in\N$, let $(e_1,\ldots,e_l)$ denote the standard basis of $\R^l$, which is orthonormal for the standard inner product $\langle\,\cdot\,,\cdot\,\rangle$. Our realizations of the simple root systems follow the convention of Bourbaki \cite{Bour}; we lay them out in Table~\ref{roots} below.

\begin{table}[htp]
\centering
\begin{tabular}{|c|c|c|}\hline
Dynkin type&\multicolumn{2}{|c|}{Positive roots $R^+$}\\\hline\hline
$\rmA_n$&$e_i-e_j$&$1\leq i<j\leq n+1$\\\hline
\multirow{3}*{$\rmB_n$}&$e_i$&$1\leq i\leq n$\\
&$e_i-e_j$&\multirow{2}*{$1\leq i<j\leq n$}\\
&$e_i+e_j$&\\\hline
\multirow{3}*{$\rmC_n$}&$e_i-e_j$&\multirow{2}*{$1\leq i<j\leq n$}\\
&$e_i+e_j$&\\
&$2e_i$&$1\leq i\leq n$\\\hline
\multirow{2}*{$\rmD_n$}&$e_i-e_j$&\multirow{2}*{$1\leq i<j\leq n$}\\
&$e_i+e_j$&\\\hline
\multirow{3}*{$\rmE_6$}&$-e_i+e_j$&\multirow{2}*{$1\leq i<j\leq 5$}\\
&$e_i+e_j$&\\
&$\tfrac12(e_8-e_7-e_6+\sum_{i=1}^5(-1)^{\nu(i)}e_i)$&$\nu(i)\in\{0,1\}$, $\sum_{i=1}^5\nu(i)$ even\\\hline
\multirow{3}*{$\rmE_7$}&$-e_i+e_j$&$1\leq i<j\leq 6$ or $(i,j)=(7,8)$\\
&$e_i+e_j$&$1\leq i<j\leq 6$\\
&$\tfrac12(e_8-e_7+\sum_{i=1}^6(-1)^{\nu(i)}e_i)$&$\nu(i)\in\{0,1\}$, $\sum_{i=1}^6\nu(i)$ odd\\\hline
\multirow{3}*{$\rmE_8$}&$-e_i+e_j$&\multirow{2}*{$1\leq i<j\leq 8$}\\
&$e_i+e_j$&\\
&$\tfrac12(e_8+\sum_{i=1}^7(-1)^{\nu(i)}e_i)$&$\nu(i)\in\{0,1\}$, $\sum_{i=1}^7\nu(i)$ even\\\hline
\multirow{4}*{$\rmF_4$}&$e_i$&$1\leq i\leq 4$\\
&$e_i-e_j$&\multirow{2}*{$1\leq i<j\leq 4$}\\
&$e_i+e_j$&\\
&$\tfrac12(e_1\pm e_2\pm e_3\pm e_4)$&\\\hline
$\rmG_2$&\multicolumn{2}{|c|}{$\alpha_1,\ \alpha_2,\ \alpha_1+\alpha_2,\ 2\alpha_1+\alpha_2,\ 3\alpha_1+\alpha_2,\ 3\alpha_1+2\alpha_2$}\\\hline
\end{tabular}
\caption{Simple root systems in the convention of \cite{Bour}. For $\rmG_2$ it is most convenient to express the roots in a basis of simple roots $\alpha_1,\alpha_2$.}
\label{roots}
\end{table}

Next, we record some auxiliary formulas that will be repeatedly used in subsequent constructions. The reader may easily verify them:
\begin{align}
\sum_{1\leq i<j\leq 2m}(-1)^{i+j}(e_i+e_j)&=-\sum_{i=1}^{2m}e_i,\label{eq:S+even}\\
\sum_{1\leq i<j\leq 2m+1}(-1)^{i+j}(e_i+e_j)&=-2\sum_{i=1}^me_{2i},\label{eq:S+odd}\\
\sum_{1\leq i<j\leq 2m}(-1)^{i+j}(e_i-e_j)&=\sum_{i=1}^{2m}(-1)^ie_i,\label{eq:S-even}\\
\sum_{1\leq i<j\leq 2m+1}(-1)^{i+j}(e_i-e_j)&=0.\label{eq:S-odd}
\end{align}

{We will also need the following technical definition.
\begin{defi}
For any vector $v\in\R^l$, its \emph{support} is the set
\[I_v:=\{i\in\{1,\ldots,l\}\,|\,\langle e_i,v\rangle\neq0\}.\]
\end{defi}}

\section{Balanced subsets of maximal cardinality}

The name ``well-balanced'' is motivated by the following observation.

\begin{lemm}
\label{lemm:maxbalanced}
A balanced subset of $R^+$ of maximal cardinality is always well-balanced.
\end{lemm}
\begin{proof}
Suppose that $S\subset R^+$ is balanced, but not well-balanced. Then there exist $\beta_1,\beta_2\in R^+\setminus S$ such that either $\beta_1+\beta_2$ or $\beta_1-\beta_2$, is a positive root. Let $\gamma\in R^+$ be this root, and let $s: S\to\{\pm1\}$ be a choice of signs such that $\sum_{\alpha\in S}s_\alpha\alpha=0$.

If $\gamma\notin S$, we can construct a balanced subset $S':=S\cup\{\beta_1,\beta_2,\gamma\}$ by extending $s$ with suitable signs,  namely $s_{\beta_1}=s_{\beta_2}=1$ and $s_{\gamma}=-1$ if $\gamma=\beta_1+\beta_2$, and $s_{\beta_1}=1$ and $s_{\beta_2}=s_{\gamma}=-1$ if $\gamma=\beta_1-\beta_2$.

If however $\gamma\in S$, we can again construct a balanced subset $S':=(S\setminus\{\gamma\})\cup\{\beta_1,\beta_2\}$ by choosing $s_{\beta_1}=s_{\beta_2}=s_\gamma$ if $\gamma=\beta_1+\beta_2$, and $s_{\beta_1}=-s_{\beta_2}=s_\gamma$ if $\gamma=\beta_1-\beta_2$.

In any case, we have found a balanced subset $S'\subset R^+$ with $|S'|>|S|$, so $S$ is not of maximal cardinality.
\end{proof}

We now proceed to the main theorem of this section. Note that the cases where $R^+$ itself is balanced were discussed in \cite{AS}; we include them for the sake of completeness.

\begin{theo}
\label{theo:maximalsize}
The minimal cocardinality of a balanced (thus well-balanced) subset of a simple root system is given in the following table.
\begin{center}
\begin{tabular}{|c||c|c|c|c|c|c|c|c|c|}
\hline
$n$&$\rmA_n$&$\rmB_n$&$\rmC_n$&$\rmD_n$&$\rmE_6$&$\rmE_7$&$\rmE_8$&$\rmF_4$&$\rmG_2$\\
\hline
$4k$&$0$&$2k$&$0$&$0$&
\multirow{4}{*}{$0$}&
\multirow{4}{*}{$3$}&
\multirow{4}{*}{$0$}&
\multirow{4}{*}{$0$}&
\multirow{4}{*}{$0$}\\
$4k+1$&$2k+1$&$2k+1$&$1$&$0$&
\multicolumn{1}{|c|}{}&
\multicolumn{1}{c|}{}&
\multicolumn{1}{c|}{}&
\multicolumn{1}{c|}{}&
\multicolumn{1}{c|}{}\\
$4k+2$&$0$&$2k+1$&$1$&$2$&
\multicolumn{1}{|c|}{}&
\multicolumn{1}{c|}{}&
\multicolumn{1}{c|}{}&
\multicolumn{1}{c|}{}&
\multicolumn{1}{c|}{}\\
$4k+3$&$2k+2$&$2k+2$&$0$&$2$&
\multicolumn{1}{|c|}{}&
\multicolumn{1}{c|}{}&
\multicolumn{1}{c|}{}&
\multicolumn{1}{c|}{}&
\multicolumn{1}{c|}{}\\
\hline
\end{tabular}
\end{center}
\end{theo}
\begin{proof}
We deal with each type of simple root system separately.
\paragraph*{Case $\rmA_n$.}
If $n$ is even, it follows from \eqref{eq:S-odd} that $R^+$ itself is balanced. Let us now assume that $n=2m+1$ is odd, and let $S\subset R^+$ be a balanced subset. For each coordinate $i=1,\ldots,2m+2$ and each root $\alpha\in R^+$, the inner product $\langle e_i,\alpha\rangle$ can be either $0$ or $\pm1$. Therefore, in order for $S$ to be balanced, there must be an even number of $\alpha\in S$ with $\langle e_i,\alpha\rangle$ odd. Since there are exactly $2m+1$ many elements $\alpha\in R^+$ with $\langle e_i,\alpha\rangle$ odd, it follows that for each $i$, at least one such element must lie outside of $S$. Now every $\alpha\in R^+$ has support of cardinality $|I_\alpha|=2$, so there must be at least $m+1$ many positive roots outside of $S$.

Conversely, \eqref{eq:S-even} states that
\[\sum_{1\leq i<j\leq 2m+2}(-1)^{i+j}(e_i-e_j)=-(e_1-e_2)-(e_3-e_4)-\ldots-(e_{2m+1}-e_{2m+2}).\]
The roots $e_1-e_2,e_3-e_4,\ldots,e_{2m+1}-e_{2m+2}$ appear with the same sign on both sides of the equation, so their complement is a balanced subset of cocardinality $m+1$,
\[\sum_{\substack{1\leq i<j\leq 2m+2\\(i,j)\neq(2k-1,2k)}}(-1)^{i+j}(e_i-e_j)=0.\]

\paragraph*{Case $\rmB_n$.}
Let $i=1,\ldots,n$. Again, $\langle e_i,\alpha\rangle$ can be either $0$ or $\pm1$ for any $\alpha\in R^+$. For each $i$, there are $2n-1$ elements $\alpha\in R^+$ with $\langle e_i,\alpha\rangle$ odd. So again, if $S\subset R^+$ is balanced, at least one such element must lie outside of $S$ for each $i$. Every $\alpha\in R^+$ satisfies $|I_\alpha|\leq2$, so $|R^+\setminus S|\geq\lceil n/2\rceil$.

To construct a balanced set of this cocardinality, we distinguish two cases. For $n=2m$ even, it follows from \eqref{eq:S+even} and \eqref{eq:S-even} that
\[\sum_{1\leq i<j\leq 2m}(-1)^{i+j}((e_i+e_j)+(e_i-e_j))+\sum_{i=1}^{2m}e_i=-(e_1-e_2)-(e_3-e_4)-\ldots-(e_{2m-1}-e_{2m}),\]
and the roots $e_1-e_2,e_3-e_4,\ldots,e_{2m-1}-e_{2m}$ appear with the same sign on both sides. Thus their complement is a balanced subset of cocardinality $m=\lceil n/2\rceil$,
\[\sum_{1\leq i<j\leq 2m}(-1)^{i+j}(e_i+e_j)+\sum_{\substack{1\leq i<j\leq 2m\\(i,j)\neq(2k-1,2k)}}(-1)^{i+j}(e_i-e_j)+\sum_{i=1}^{2m}e_i=0.\]
If $n=2m+1$ is odd, we note that
\[
\begin{split}
    \sum_{1\leq i<j\leq 2m+1}(-1)^{i+j}((e_i+e_j)&-(e_i-e_j))+\sum_{i=1}^{2m+1}e_i\\
    &=(e_1-e_2)+(e_3-e_4)+\ldots+(e_{2m-1}-e_{2m})+e_{2m+1}
    \end{split}\]
by \eqref{eq:S+odd} and \eqref{eq:S-odd}, and the roots $e_1-e_2,e_3-e_4,\ldots,e_{2m-1}-e_{2m},$ and $e_{2m+1}$ appear with the same sign on both sides. Thus their complement is a balanced subset of cocardinality $m+1=\lceil n/2\rceil$:
\[\sum_{1\leq i<j\leq 2m+1}(-1)^{i+j}(e_i+e_j)+\sum_{\substack{1\leq i<j\leq 2m+1\\(i,j)\neq(2k-1,2k)}}(-1)^{i+j+1}(e_i-e_j)+\sum_{i=1}^{2m}e_i=0.\]

\paragraph*{Case $\rmC_n$.}
Given any signed combination of $R^+$, no matter what the signs of $e_i+e_j$ and $e_i-e_j$ are, the sum of these two terms in the signed combination is of the form $\pm 2e_k$ for $k=i$ or $j$. By grouping the pairs $e_i+e_j$ and $e_i-e_j$, we obtain $n(n-1)/2$ terms of the form $\pm 2e_k$ in any signed combination of these roots. Together with the remaining roots $2e_1,\ldots 2e_n\in R^+$, any signed combination of $R^+$ can be written as a sum of $n(n+1)/2$ terms of the form $\pm 2e_k$, which cannot all cancel out if $n(n+1)/2$ is odd, i.e.~if $n\equiv 1,2\mod4$. In this case, any balanced subset must have cocardinality at least $1$.

We turn to constructing the optimal balanced subsets. For $n=2m$ even, one may verify with \eqref{eq:S+odd} and \eqref{eq:S-odd} that
\begin{align*}
&\sum_{1\leq i<j\leq 2m-1}(-1)^{i+j}((e_i+e_j)+(e_i-e_j))+\sum_{i=1}^{2m-1}(-1)^i2e_i\\
&\quad+\sum_{i=1}^m((e_{2i-1}-e_{2m})+(e_{2i-1}+e_{2m}))+\sum_{i=1}^{m-1}(-1)^i((e_{2i}-e_{2m})-(e_{2i}+e_{2m}))\\
&=(1+(-1)^m)e_{2m}=\begin{cases}
    2e_{2m},&m\text{ even,}\\
    0,&m\text{ odd}
\end{cases}
\end{align*}
is a signed combination of $R^+\setminus\{2e_{2m}\}$. This gives a balanced subset of cocardinality $1$ if $n\equiv2\mod4$, while adding in $-2e_{2m}$ shows that $R^+$ is balanced if $n\equiv0\mod4$.

If $n=2m+1$ is odd, \eqref{eq:S+even} and \eqref{eq:S-even} imply that similarly,
\begin{align*}
&\sum_{1\leq i<j\leq 2m}(-1)^{i+j}((e_i+e_j)+(e_i-e_j))+\sum_{i=1}^{2m}(-1)^{i+1}2e_i\\
&\quad+\sum_{i=1}^m((e_{2i}-e_{2m+1})+(e_{2i}+e_{2m+1}))+\sum_{i=1}^{m}(-1)^i((e_{2i-1}-e_{2m+1})-(e_{2i-1}+e_{2m+1}))\\
&=(1+(-1)^{m+1})e_{2m+1}=\begin{cases}
    0,&m\text{ even,}\\
    2e_{2m+1},&m\text{ odd}
\end{cases}
\end{align*}
is a signed combination of $R^+\setminus\{2e_{2m+1}\}$, showing that there exists a balanced subset of cocardinality $1$ if $n\equiv1\mod4$, and (after adding in $-2e_{2m+1}$) that $R^+$ is itself balanced if $n\equiv3\mod4$.

\paragraph*{Case $\rmD_n$.}
As in the previous case, given any signed combination of $R^+$, we may group the pairs $e_i+e_j$ and $e_i-e_j$ to obtain a sum of $n(n-1)/2$ terms of the form $\pm 2e_k$, which cannot all cancel out if $n(n-1)/2$ is odd, i.e.~if $n\equiv 2,3\mod4$. However, for each $i=1,\ldots,n$, the inner product $\langle e_i,\alpha\rangle$ can again only be $0$ or $\pm1$, and there are $2(n-1)$ many $\alpha\in R^+$ with $\langle e_i,\alpha\rangle$ odd. Thus, if $S\subset R^+$ is balanced, an even number of such elements must lie outside $S$. In particular it follows that $|R^+\setminus S|\geq2$ if $n\equiv 2,3\mod4$.

The optimal balanced subsets are swiftly constructed. With \eqref{eq:S+even}--\eqref{eq:S-odd}, one can check that
\begin{align}
&\sum_{1\leq i<j\leq2m-1}(-1)^{i+j}((e_i-e_j)+(e_i+e_j))+\sum_{i=1}^{m-1}((e_{2i}-e_{2m})+(e_{2i}+e_{2m}))\notag\\
&\quad+\sum_{i=1}^{m}(-1)^i((e_{2i-1}-e_{2m})-(e_{2i-1}+e_{2m}))=(1-(-1)^m)e_{2m},\label{eq:Deven}\\
&\sum_{1\leq i<j\leq2m}(-1)^{i+j}((e_i-e_j)+(e_i+e_j))+\sum_{i=1}^{m}((e_{2i-1}-e_{2m+1})+(e_{2i-1}+e_{2m+1}))\notag\\
&\quad+\sum_{i=1}^{m}(-1)^i((e_{2i}-e_{2m+1})-(e_{2i}+e_{2m+1}))=(1-(-1)^m)e_{2m+1}.\label{eq:Dodd}
\end{align}
This shows that $R^+$ is balanced if $n\equiv0,1\mod4$. For $n\equiv2\mod4$, we may remove the terms $e_{2m-1}+e_{2m}$ and $-(e_{2m-1}-e_{2m})$ from the left hand side of \eqref{eq:Deven} to obtain the balanced subset $R^+\setminus\{e_{2m-1}\pm e_{2m}\}$ of cocardinality $2$. Similarly, if $n\equiv3\mod4$, we see that $R^+\setminus\{e_{2m}\pm e_{2m+1}\}$ is a balanced subset of cocardinality $2$ by removing $e_{2m}+e_{2m+1}$ and $-(e_{2m}-e_{2m+1})$ from \eqref{eq:Dodd}.

\paragraph*{Case $\rmE_6$.}
In order to show that $R^+$ is balanced, we split it into three disjoint subsets, each of which is balanced itself. First, the roots of the form $\pm e_i+e_j$ are (up to sign) also the positive roots of $\rmD_5$, which we have just shown to be a balanced set. Second, for any $\nu:\{1,\ldots,4\}\to\{0,1\}$, denote
\[\alpha_{\pm,\nu}:=\tfrac12(e_8-e_7-e_6\pm e_5+\sum_{i=1}^4(-1)^{\nu(i)}e_i)\]
and consider the $16$-element subsets
\begin{align*}
S_+&:=\left\{\alpha_{+,\nu}\,\middle|\,\nu(i)\in\{0,1\},\ \sum_{i=1}^4\nu(i)\text{ even}\right\},\\
S_-&:=\left\{\alpha_{-,\nu}\,\middle|\,\nu(i)\in\{0,1\},\ \sum_{i=1}^4\nu(i)\text{ odd}\right\}
\end{align*}
of $R^+$. The map $\alpha_{\pm,\nu}\mapsto\alpha_{\pm,1-\nu}$ is a fixed-point-free involution on $S_\pm$, splitting $S_\pm$ into $8$ pairs with
\[\alpha_{\pm,\nu}+\alpha_{\pm,1-\nu}=e_8-e_7-e_6\pm e_5.\]
By choosing $+$ signs on $4$ of these pairs, and $-$ on the $4$ remaining ones, we find a signed combination of $S_\pm$ that vanishes.

\paragraph*{Case $\rmE_7$.}
Let $\rho=\tfrac12\sum_{\alpha\in R^+}\alpha$ denote the half-sum of positive roots. We record \cite{Bour}
\[2\rho=(0,2,4,6,8,10,-17,17).\]
First, the vector $v:=\tfrac12\sum_{i=1}^8e_i$ has inner product $0$ or $\pm1$ with each $\alpha\in R^+$. Thus, for any choice of signs $s:R^+\to\{\pm1\}$, 
\[\langle v,\sum_{\alpha\in R^+}s_\alpha\alpha\rangle\equiv\langle v,2\rho\rangle\equiv1\mod2,\]
and in particular $R^+$ cannot be balanced.

Since in our case, all roots have squared length $2$, we have $\langle\alpha,\beta\rangle\in\Z$ for all $\alpha,\beta\in R^+$ \cite[\S VI.1.1]{Bour}. In fact, by \cite[\S VI.1.3, Prop.~8]{Bour}, $\langle\alpha,\beta\rangle\in\{-2,-1,0,1,2\}$, with $\pm2$ only occurring if $\alpha=\pm\beta$. Moreover, by \cite[\S VI.1.10, Prop.~29]{Bour}, $\langle\beta,2\rho\rangle=|\beta|^2$ if $\beta$ is a simple root, which implies that $\langle\beta,\rho\rangle\in\Z$ for any $\beta\in R^+$. Thus
\[\langle\beta,\sum_{\alpha\in R^+}s_\alpha\alpha\rangle\equiv2\langle\beta,\rho\rangle\equiv0\mod2.\]
If there existed $\alpha_1\in R^+$ such that $R^+\setminus\{\alpha_1\}$ is balanced, then by the above $\langle\beta,\alpha_1\rangle\equiv0\mod2$ for all $\beta\in R^+$. This means that $\alpha_1\perp\beta$ for all $\beta\in R^+\setminus\{\alpha_1\}$, which is impossible.

Next, assume that $R^+\setminus\{\alpha_1,\alpha_2\}$ is balanced, with $\alpha_1\neq\alpha_2$. Then the above argument shows that
\begin{equation}
\langle\beta,\alpha_1\rangle\equiv\langle\beta,\alpha_2\rangle\mod2\label{eq:E7tworoots}
\end{equation}
for all $\beta\in R^+$. This amounts to saying that any $\beta\in R^+$ is equal or perpendicular to $\alpha_1$ if and only if it is equal or perpendicular to $\alpha_2$. In particular, by setting $\beta=\alpha_1$ in \eqref{eq:E7tworoots}, we find $\langle\alpha_1,\alpha_2\rangle\equiv0\mod2$, hence $\alpha_1\perp\alpha_2$. Now we need to distinguish a few cases.

Suppose that one of the two roots, say $\alpha_1$, is $-e_7+e_8$. Then $\alpha_1\perp\alpha_2$ implies that $\alpha_2=\pm e_i+e_j$ for some indices $1\leq i<j\leq 6$. Now for any $1\leq k\leq 6$ with $k\neq i,j$, the positive root $e_i+e_k$ is perpendicular to $\alpha_1$ but not to $\alpha_2$. So by \eqref{eq:E7tworoots}, this case is impossible.

Taking now $\beta=-e_7+e_8$ in \eqref{eq:E7tworoots}, we see that either $\langle\beta,\alpha_1\rangle=\langle\beta,\alpha_2\rangle=0$ and both $\alpha_1,\alpha_2$ are of the form $\pm e_i+e_j$, $1\leq i<j\leq6$, or $\langle\beta,\alpha_1\rangle=\langle\beta,\alpha_2\rangle=1$ and both $\alpha_1,\alpha_2$ are of the form
\[\alpha_\nu:=\tfrac12(-e_7+e_8+\sum_{i=1}^6(-1)^{\nu(i)}e_i)\]
for some $\nu:\{1,\ldots,6\}\to\{0,1\}$ with $\sum_i\nu(i)$ odd.

In the first case, let $\alpha_1=\pm e_i+e_j$. Then it follows from \eqref{eq:E7tworoots} applied to $\beta$ of the same type that $\alpha_2=\mp e_i+e_j$. But then \eqref{eq:E7tworoots} with $\beta=\alpha_\nu$ fails for every $\nu$, since $\alpha_\nu$ is always orthogonal to exactly one of $\alpha_1,\alpha_2$.

In the second case, write $\alpha_1=\alpha_\mu$ and $\alpha_2=\alpha_\nu$. Then $\nu\neq1-\mu$, since
\[\langle\alpha_\mu,\alpha_{1-\mu}\rangle=-1\neq0.\]
Thus there exist $1\leq i<j\leq 6$ such that $\mu(i)=\nu(i)$ and $\mu(j)\neq\nu(j)$. But then \eqref{eq:E7tworoots} for $\beta=e_i+e_j$ gives a contradiction, since $\beta$ is orthogonal to exactly one of $\alpha_1,\alpha_2$.

In conclusion, $R^+\setminus\{\alpha_1,\alpha_2\}$ cannot be balanced, and any balanced subset of $R^+$ must have cocardinality at least $3$.

We can construct a balanced subset with this cocardinality. Like before, we split $R^+$ into several parts. First, the map $\alpha_\nu\mapsto\alpha_{1-\nu}$ is again a fixed-point-free involution on the set
\[S':=\left\{\alpha_\nu\,\middle|\,\nu(i)\in\{0,1\},\ \sum_{i=1}^6\nu(i)\text{ odd}\right\}.\]
This splits $S'$ into $32$ pairs with $\alpha_\nu+\alpha_{1-\nu}=-e_7+e_8$, and as before we see that $S'$ is balanced.

Second, the roots $\pm e_i+e_j$, $1\leq i<j\leq 6$, are (up to sign) the positive roots of $\rmD_6$, and we have already shown that this root system has a balanced subset $S''$ of cocardinality $2$. Excluding the last remaining positive root $-e_7+e_8$, we finally obtain a balanced subset $S:=S'\cup S''\subset R^+$ with $|R^+\setminus S|=3$.

\paragraph*{Case $\rmE_8$.}
Similar to the case $\rmE_6$, we split $R^+$ into three disjoint balanced subsets. The roots $\pm e_i+e_j$ are (up to sign) the positive roots of $\rmD_8$, which we have seen to be balanced. We then define
\[\alpha_{\pm,\nu}:=\tfrac12(e_8\pm e_7+\sum_{i=1}^6(-1)^{\nu(i)}e_i)\]
for any $\nu:\{1,\ldots,6\}\to\{0,1\}$, and the $64$-element subsets
\begin{align*}
S_+&:=\left\{\alpha_{+,\nu}\,\middle|\,\nu(i)\in\{0,1\},\ \sum_{i=1}^6\nu(i)\text{ even}\right\},\\
S_-&:=\left\{\alpha_{-,\nu}\,\middle|\,\nu(i)\in\{0,1\},\ \sum_{i=1}^6\nu(i)\text{ odd}\right\}.
\end{align*}
The analogous argument as in the case $\rmE_6$ shows that the $S_\pm$ are balanced.

\paragraph*{Case $\rmF_4$.}
An explicit signed combination showing that $R^+$ is balanced is exhibited in the proof of \cite[Thm.~6]{AS} (their convention for the root system coincides with ours in this particular case).

\paragraph*{Case $\rmG_2$.}
$R^+$ is balanced since
\begin{align*}
\alpha_1+(\alpha_1+\alpha_2)-(2\alpha_1+\alpha_2)+\alpha_2+(3\alpha_1+\alpha_2)-(3\alpha_1+2\alpha_2)=0.
\end{align*}
\end{proof}

\begin{rema}
One may also ask about balanced subsets of $R^+$ which are maximal under inclusion. They do not necessarily have the same cardinality.

For example, in a root system of type $\rmC_5$, start with the balanced subset $R^+\setminus\{2e_5\}$. In the proof of Theorem~\ref{theo:maximalsize} we have constructed a vanishing signed combination of $R^+\setminus\{2e_5\}$ containing the two terms $-(e_1-e_5)+(e_1+e_5)$. Replacing them by $2e_5$, we obtain a vanishing signed combination of $R^+\setminus\{e_1+e_5,e_1-e_5\}$. This is maximal among balanced sets: indeed, note that $\langle e_5,\alpha\rangle\in\{-2,-1,0,1,2\}$ for any $\alpha\in R^+$, and there are $8$ elements $\alpha\in R^+$ with $\langle e_5,\alpha\rangle=\pm1$. Thus for any choice of signs,
\[\langle e_5,\sum_{\alpha\in R^+\setminus\{e_1+e_5\}}s_\alpha\alpha\rangle\equiv1\mod2,\]
and the same holds for $R^+\setminus\{e_1-e_5\}$.

A classification of the maximal balanced subsets of the sets of positive roots of simple root systems seems out of reach for now.
\end{rema}

\section{Well-balanced subsets of minimal cardinality}

A balanced subset of $R^+$ can be very small: of course, the empty set is balanced, and (except for rank $1$) one can also always find balanced subsets of cardinality $3$. It more interesting to ask about the smallest possible cardinality of a \emph{well-balanced} subset. Since strongly orthogonal roots are in particular orthogonal, a crude upper bound on the cocardinality of a well-balanced subset of $R^+$ is the rank of $R$. An improvement over this are the sharp upper bounds on the cardinality of strongly orthogonal subsets in \cite[Thm.~3.1, Thm.~5.1]{AK}. We will use these facts in the proof of the following result.

\begin{theo}
\label{theo:minimalsize}
The maximal cocardinality of a well-balanced subset of a simple root system is given in the following table.
\begin{center}
\begin{tabular}{|c||c|c|c|c|c|c|c|c|c|}
\hline
$n$&$\rmA_n$&$\rmB_n$&$\rmC_n$&$\rmD_n$&$\rmE_6$&$\rmE_7$&$\rmE_8$&$\rmF_4$&$\rmG_2$\\
\hline
$4k$&$0$&$2k$&$4k$&$4k$&
\multirow{4}{*}{$4$}&
\multirow{4}{*}{$7$}&
\multirow{4}{*}{$8$}&
\multirow{4}{*}{$4$}&
\multirow{4}{*}{$2$}\\
$4k+1$&$2k+1$&$2k+1$&$4k+1$&$4k$&
\multicolumn{1}{|c|}{}&
\multicolumn{1}{c|}{}&
\multicolumn{1}{c|}{}&
\multicolumn{1}{c|}{}&
\multicolumn{1}{c|}{}\\
$4k+2$&$0$&$2k+1$&$4k+1$&$4k+2$&
\multicolumn{1}{|c|}{}&
\multicolumn{1}{c|}{}&
\multicolumn{1}{c|}{}&
\multicolumn{1}{c|}{}&
\multicolumn{1}{c|}{}\\
$4k+3$&$2k+2$&$2k+2$&$4k+2$&$4k+2$&
\multicolumn{1}{|c|}{}&
\multicolumn{1}{c|}{}&
\multicolumn{1}{c|}{}&
\multicolumn{1}{c|}{}&
\multicolumn{1}{c|}{}\\
\hline
\end{tabular}
\end{center}
\end{theo}
\begin{proof}
Again, we proceed case by case.
\paragraph*{Case $\rmA_n$.}
Two distinct roots of $\rmA_n$ are strongly orthogonal if and only if their supports are disjoint, and every $\alpha\in R^+$ has $|I_\alpha|=2$. This immediately implies that the maximal cardinality of a strongly orthogonal subset is $\lfloor(n+1)/2\rfloor$, cf.~\cite{AK}. For $n$ odd, this is exactly the cocardinality of the well-balanced subset constructed in Theorem~\ref{theo:maximalsize}.

Now let $n=2m$ be even. For each $i=1,\ldots,2m+1$ and $\alpha\in R^+$, the inner product $\langle e_i,\alpha\rangle$ can take values $0$ or $\pm1$, and there are exactly $2m$ many $\alpha\in R^+$ with $\langle e_i,\alpha\rangle$ odd. Thus, if $S\subset R^+$ is a well-balanced subset, an even number of such roots must lie outside $S$. But for each $i$, a strongly orthogonal subset contains at most one element $\alpha$ with $\langle e_i,\alpha\rangle\neq0$, since any two such roots are not strongly orthogonal. Thus the only possibility is that no such element lies outside $S$ for each $i$, hence $S=R^+$. And indeed, by \eqref{eq:S-odd}, $R^+$ is well-balanced.

\paragraph*{Case $\rmB_n$.}
Again, for each $i=1,\ldots,n$ and $\alpha\in R^+$, the inner product $\langle e_i,\alpha\rangle$ can take values $0$ or $\pm1$, and there are exactly $2n-1$ many $\alpha\in R^+$ such that $\langle e_i,\alpha\rangle$ is odd. Therefore, if $S$ is a balanced subset of $R^+$, then for each $i$ the number of roots $\alpha\in R^+\setminus S$ such that $\langle e_i,\alpha\rangle\neq0$ is odd.

On the other hand, for each $i$, the maximal cardinality of a strongly orthogonal set of positive roots whose support contains $i$ is 2. Consequently, if $S\subset R^+$ is well-balanced, then for every $i$ there is exactly one root $\alpha\in R^+\setminus S$ such that $\langle e_i,\alpha\rangle\neq0$. Since moreover $e_i$ and $e_j$ are not strongly orthogonal for $i\neq j$, it follows that $ R^+\setminus S$ contains at most one root of this form, and all others have a support of cardinality 2. This implies that $|R^+\setminus S|=\lceil n/2\rceil$. This is exactly the cocardinality of the well-balanced subset constructed in Theorem~\ref{theo:maximalsize}.

\paragraph*{Case $\rmC_n$.}
First, the cocardinality of a well-balanced subset $S\subset R^+$ is bounded by the rank, $|R^+\setminus S|\leq n$. If $n\equiv0,1\mod4$, this bound is attained: the positive roots of the subsystem $\rmD_n\subset\rmC_n$ form a balanced set by Theorem~\ref{theo:maximalsize}, and the remaining roots $2e_1,\ldots,2e_n$ are strongly orthogonal.

Now consider $n\equiv2,3\mod4$. One quickly verifies that two distinct roots of $\rmC_n$ are strongly orthogonal if and only if their supports are disjoint. If $S\subset R^+$ is well-balanced, its complement can only have $|R^+\setminus S|=n$ if each root $\alpha\in R^+\setminus S$ has $|I_\alpha|=1$, i.e.~if $R^+\setminus S=\{2e_1,\ldots,2e_n\}$. But this would mean that the positive roots of $\rmD_n$ are a balanced set, contradicting Theorem~\ref{theo:maximalsize}.

However, there exists a well-balanced set of cocardinality $n-1$. The proof of Theorem~\ref{theo:maximalsize}, case $\rmD_n$, shows that $R^+\backslash\{e_{n-1}+e_n,e_{n-1}-e_n,2e_1,\ldots,2e_n\}$ is balanced. Since $(e_{n-1}+e_n)-(e_{n-1}-e_n)=2e_n$, we may add these three roots to the previous balanced set and obtain the well-balanced subset $R^+\backslash\{2e_1,\ldots,2e_{n-1}\}$ of cocardinality $n-1$.

\paragraph*{Case $\rmD_n$.}
\cite{AK} tells us that a strongly orthogonal subset of $R^+$ can have cardinality at most $2\lfloor n/2\rfloor$.
This is the cardinality of the strongly orthogonal subset $P:=\{e_1+e_2, e_1-e_2, e_3+e_4, e_3-e_4, ...\}$. We will show that $R^+\setminus P$ is balanced.

First, if $n=2m$ is even, we have by \eqref{eq:S+odd} and \eqref{eq:S-odd}
\begin{align*}
\sum_{\substack{1\leq i<j\leq 2m-1\\(i,j)\neq(2k-1,2k)}}(-1)^{i+j}((e_i-e_j)+(e_i+e_j))&=\sum_{1\leq i<j\leq 2m-1}(-1)^{i+j}((e_i-e_j)+(e_i+e_j))\\
&\quad+\sum_{i=1}^{m-1}((e_{2i-1}+e_{2i})+(e_{2i-1}-e_{2i}))\\
&=-2\sum_{i=1}^{2m-2}(-1)^ie_i,
\end{align*}
and consequently
\begin{align*}
\sum_{\substack{1\leq i<j\leq 2m-1\\(i,j)\neq(2k-1,2k)}}(-1)^{i+j}((e_i-e_j)+(e_i+e_j))+\sum_{i=1}^{2m-2}(-1)^i((e_i+e_{2m})+(e_i-e_{2m}))=0
\end{align*}
is a vanishing signed combination of $R^+\setminus P$. For $n=2m+1$ odd, \eqref{eq:S+even} and \eqref{eq:S-even} show that
\begin{align*}
\sum_{\substack{1\leq i<j\leq 2m\\(i,j)\neq(2k-1,2k)}}(-1)^{i+j}((e_i-e_j)+(e_i+e_j))&=\sum_{1\leq i<j\leq 2m}(-1)^{i+j}((e_i-e_j)+(e_i+e_j))\\
&\quad+\sum_{i=1}^m((e_{2i-1}+e_{2i})+(e_{2i-1}-e_{2i}))\\
&=0,
\end{align*}
so we find that
\begin{align*}
\sum_{\substack{1\leq i<j\leq 2m-1\\(i,j)\neq(2k-1,2k)}}(-1)^{i+j}((e_i-e_j)+(e_i+e_j))+\sum_{i=1}^{2m}(-1)^i((e_i+e_{2m+1})-(e_i-e_{2m+1}))=0
\end{align*}
is again a vanishing signed combination of $R^+\setminus P$. In total, there is always a well-balanced set of cocardinality $2\lfloor n/2\rfloor$.

\paragraph*{Case $\rmE_6$.}
By \cite{AK}, a strongly orthogonal subset of $R^+$ has cardinality at most $4$. Indeed, by the previous case, there is a well-balanced subset of the positive roots of $\rmD_5$ with cocardinality $4$. The remaining roots of $\rmE_6$ form a balanced set, as shown in the proof of Theorem~\ref{theo:maximalsize}, and the sum or difference of two roots of the form $\pm e_i+e_j$ is never of the form $\alpha_{\pm,\nu}$. Thus, taking the union of the two balanced subsets, we obtain a well-balanced subset of $R^+$ with cocardinality $4$.

\paragraph*{Case $\rmE_7$.}
The positive roots of $\rmE_7$ split into the positive roots (up to sign) of $\rmD_6$, the root $-e_7+e_8$, and the balanced set $S'$ (see proof of Theorem~\ref{theo:maximalsize}). We have seen that $\rmD_6$ has a well-balanced set of positive roots of cocardinality $6$. Taking the union with $S'$, we find a well-balanced subset of $R^+$ with cocardinality $7$, since the sum or difference of two $\pm e_i+e_j$, $1\leq i,j\leq 6$, can neither be $-e_7+e_8$ nor in $S'$. Since the cocardinality of well-balanced sets is bounded above by the rank of $\rmE_7$, the constructed set is optimal.

\paragraph*{Case $\rmE_8$.}
Similarly, $R^+$ splits into the positive roots (up to sign) of $\rmD_8$ and the two balanced sets $S_\pm$. Since $\rmD_8$ has a well-balanced subset of cocardinality $8$, the union with $S_+\cup S_-$ gives a well-balanced subset of $R^+$ of cocardinality $8$, which is the rank of $\rmE_8$ and hence optimal.

\paragraph*{Case $\rmF_4$.}
The four roots $e_1\pm e_2,e_3\pm e_4$ are clearly strongly orthogonal. We show that $R^+\setminus\{e_1\pm e_2,e_3\pm e_4\}$ is a balanced subset; this is optimal because $4$ is the rank of $\rmF_4$. First, as in case $\rmD_n$, we note that
\[\sum_{\substack{1\leq i<j\leq 4\\(i,j)\neq(2k-1,2k)}}(-1)^{i+j}((e_i-e_j)+(e_i+e_j))=0.\]
We partition the set $\{\tfrac12 (e_1\pm e_2\pm e_3\pm e_4)\}\subset R^+$ into four subsets, according to the number of $-$ signs appearing:
\begin{align*}
S_0&=\{\tfrac12(e_1+e_2+e_3+e_4)\},\\
S_1&=\{\tfrac12(e_1-e_2+e_3+e_4),\tfrac12(e_1+e_2-e_3+e_4),\tfrac12(e_1+e_2+e_3-e_4)\},\\
S_2&=\{\tfrac12(e_1+e_2-e_3-e_4),\tfrac12(e_1-e_2+e_3-e_4),\tfrac12(e_1-e_2-e_3+e_4)\},\\
S_3&=\{\tfrac12(e_1-e_2-e_3-e_4)\}.
\end{align*}
Then one quickly verifies that $\sum_{\alpha\in S_0\cup S_1\cup S_3}\alpha-\sum_{\alpha\in S_2}\alpha=e_1+e_2+e_3+e_4$. In total,
\begin{align*}
\sum_{\substack{1\leq i<j\leq 4\\(i,j)\neq(2k-1,2k)}}(-1)^{i+j}((e_i-e_j)+(e_i+e_j))-\sum_{i=1}^4e_i+\sum_{\alpha\in S_0\cup S_1\cup S_3}\alpha-\sum_{\alpha\in S_2}\alpha=0.
\end{align*}

\paragraph*{Case $\rmG_2$.}
An explicit signed combination is
\[\alpha_1-(\alpha_1+\alpha_2)-(3\alpha_1+\alpha_2)+(3\alpha_1+2\alpha_2)=0.\]
The remaining two roots $\alpha_2$ and $2\alpha_1+\alpha_2$ are strongly orthogonal, and since the rank of $\rmG_2$ is $2$ there can be no larger set of strongly orthogonal roots.
\end{proof}

\begin{rema}
By comparing the tables in Theorems~\ref{theo:minimalsize} and~\ref{theo:maximalsize}, we see that if $R$ is of type $\rmA_n$ or $\rmB_n$, all well-balanced subsets have the same size. In particular the converse of Lemma~\ref{lemm:maxbalanced} is true only for these root systems.
\end{rema}

\end{document}